\documentclass[12pt]{article}

\usepackage{amssymb}
\usepackage{amsmath}
\usepackage{datetime}
\usepackage{hyperref}
\usepackage{subfig}
\usepackage{graphicx}
\usepackage{verbatim}
\usepackage{rotating}
\usepackage{amsthm}
\usepackage{color}
\usepackage[final]{pdfpages}
\usepackage{bbm}

\newcommand{\p}{\partial}

\newcommand{\norm}[1]{\| {#1} \|}

\newcommand{\abs}[1]{\left| {#1} \right|}

\newcommand{\Keywords}[1]{\vspace{0.2in}\par\noindent {\bf Keywords :} {#1}}

\newcommand{\R}{\mathbb{R}}
\newcommand{\N}{\mathbb{N}}

\newcommand{\expectation}{\operatorname{E}}
\newcommand{\probab}{\operatorname{P}}

\newtheorem{theorem}{Theorem}

\newtheorem{lemma}{Lemma}

\usepackage{fancyvrb} 

\DefineVerbatimEnvironment{code}{Verbatim}{fontsize=\small}

\date{}

\title{On weak convergence of stochastic heat equation with colored noise\footnote{This research was supported in part by the NSFs grant DMS-1307470}}

\author{Pavel Bezdek\footnote{Department of Mathematics, University of Utah, Salt Lake City, UT 84112-0090; E-mail: bezdek@math.utah.edu}}

\begin{document}

\maketitle

\begin{abstract}
In this work we are going to show weak convergence of a probability measure corresponding  to the solution of the following nonlinear stochastic heat equation $\frac{\p}{\p t} u_{t}(x) = \frac{\kappa}{2} \Delta u_{ t}(x) + \sigma(u_{t}(x))\eta_\alpha$ with colored noise $\eta_\alpha$ to the measure corresponding to the solution of the same equation but with white noise $\eta$ as $\alpha \uparrow 1$ on the space of continuous functions with compact support. The noise $\eta_\alpha$ is assumed to be colored in space and its covariance is given by $\expectation \left[ \eta_\alpha(t,x) \eta_\alpha(s,y) \right] = \delta(t-s) f_\alpha(x-y)$ where $f_\alpha$ is the Riesz kernel $f_\alpha(x) \propto 1/\abs{x}^\alpha$.  We will also state a result about continuity of measure in $\alpha$, for $\alpha \in (0,1)$. We will work with the classical notion of weak convergence of measures.
\Keywords{The Stochastic heat equation, colored noise, Riesz kernel}
\end{abstract}

\section{Introduction}

Throughout this work we will consider the following one-dimensional heat equation

\begin{align}\label{eq:heateqncolornoise}
\frac{\p}{\p t} u_{\alpha,t}(x) = \frac{\kappa}{2} \Delta u_{\alpha, t}(x) + \sigma(u_{\alpha,t}(x))\eta_\alpha~,~x\in \R~,~ t\geq 0,
\end{align}
with colored noise $\eta_\alpha$ having a particular covarinace structure
\begin{align}\label{eq:noisecolor}
&\expectation \left[ \eta_\alpha(t,x) \eta_\alpha(s,y) \right] = \delta(t-s) f_\alpha(x-y),
\end{align}
where
\begin{align}
f_\alpha(x) &= c_{1-\alpha} g_\alpha(x) = \hat{g}_{1-\alpha}(x)~~,~~g_\alpha(x) = \frac{1}{\abs{x}^\alpha}~\text{ for }~\alpha \in (0,1), \label{eq:covariancefunction}
\end{align}
and the constant $c_\alpha$ is 
\begin{align}
c_\alpha &= 2\frac{\sin\left( \frac{\alpha\pi}{2}\right) \Gamma(1-\alpha)}{\left( 2 \pi \right)^{1-\alpha}}.
\end{align}
The initial condition $u_{\alpha,0}(x)$ is taken to be bounded and measurable. We will also assume $\sigma$ to be Lipschitz continuous with Lipschitz constant $K$. Stochastic PDEs such as \eqref{eq:heateqncolornoise} have been studied in \cite{dalangmainsrc,ripplsturm,tudorcolorpaper,muellertribe,davarchaoticstochastic} and others. 

The function $f_\alpha$ can be thought of as an \emph{`approximation'} to the delta function in the following special sense, we know that one-dimensional Fourier transform of $g_{1-\alpha}$ , denoted by $\hat{g}_{1-\alpha}$, is equal to $f_\alpha$. We also know that the Fourier transform of a constant is $\delta$ distribution. Observe that $g_{1-\alpha}$ converges pointwise to $1$ as $\alpha \uparrow 1$. We will study the solution of \eqref{eq:heateqncolornoise}  as a function of $\alpha$. This arises noticeably in \cite[Sec. 7]{balanconus} where the authors have shown that $L^2(P)$ norm of $u_{\alpha,t}(x)$ converges to $L^2(P)$ norm of the solution to \eqref{eq:heateqnwhitenoise} as $\alpha \uparrow 1$ for every $t>0,x\in \R$ and $\sigma(x) = x$.

 The main question that has motivated this work, is whether the solution of \eqref{eq:heateqncolornoise} converges [\emph{in the appropriate sense}] to the solution of the same equation, but with white noise $\eta$ instead of colored noise $\eta_\alpha$ as $\alpha \uparrow 1$. By that we mean, the solution to 
 \begin{align}\label{eq:heateqnwhitenoise}
\frac{\p}{\p t} u_t(x) &= \frac{\kappa}{2} \Delta u_t(x) + \sigma(u_t(x))\eta~,~x\in \R~,~ t\geq 0,\\ 
u_0(x) &= u_{\alpha,0} \text{ bounded and measurable},
\end{align}
where $\eta$ denotes white noise.  We will state the main theorem in terms of measures corresponding to solutions. Let $\mathcal{C} =\mathcal{C}([0,T]\times [-N,N])$ for any fixed compact set $[0,T]\times [-N,N] \subset \R^+_0 \times \R$. Denote $P_\alpha$, the measure corresponding to $u_\alpha$,
\begin{align*}
\probab_\alpha( A) :=\begin{cases}
 \probab\left\{u_\alpha \in A \right\} ~&\text{ for } ~ \alpha \in (0,1),\\
 \probab\left\{u \in A \right\} ~&\text{ for } ~ \alpha = 1,
 \end{cases}
\end{align*}
for any Borel set $A$ of space $\mathcal{C}$. Here is the main theorem:

\begin{theorem}\label{thm:mainthm}
Measure $\probab_\alpha$ is continuous in $\alpha$, for $\alpha\in (0,1]$. We precisely mean that
$\probab_\alpha$ converges weakly to $\probab_1$ as $\alpha \uparrow 1$
 and 
$\probab_\alpha$ converges weakly to $\probab_{\alpha_0}$ as $\alpha \rightarrow \alpha_0$ for any $\alpha_0 \in (0,1)$.
\end{theorem}

The notion of weak convergence in Theorem \ref{thm:mainthm} is the classical one \cite{billingsleyconvergence}. Theorem \ref{thm:mainthm} gives us a new way of thinking about the stochastic heat equation with white noise. Instead of studying the solution to \eqref{eq:heateqnwhitenoise} we can study the solution to \eqref{eq:heateqncolornoise} for $\alpha \approx 1$. Also note, that the noise with Riesz kernel spatial covariance produces noise which is less regular. We like to think that this \emph{`roughness'} better captures properties of the stochastic heat equation with white noise.

 Before we begin the proof of Theorem \ref{thm:mainthm}, let us recall that [\emph{mild}] solutions to \eqref{eq:heateqncolornoise} and \eqref{eq:heateqnwhitenoise} are interpreted as solutions of the following integral equations \cite{davarminicourse}

\begin{align*}
u_{\alpha,t} (y) &=(u_{\alpha,0}*p_{t})(y) +\int_0^t \int_\R p_{t-s}(x-y)\sigma(u_{\alpha,s}(x)) \eta_\alpha(ds,dx),\\
u_t (y) &=(u_0*p_{t})(y) +\int_0^t \int_\R p_{t-s}(x-y)\sigma(u_s(x)) \eta(ds,dx),
\end{align*}
where $p_t$ is the heat kernel
\begin{align*}
p_t(x) = \frac{1}{\sqrt{2 \pi \kappa t}} \exp\left( - \frac{ x^2}{2\kappa t} \right) .
\end{align*}

\section{Proof of Theorem \ref{thm:mainthm}} \label{sec:mainsectionwithproof}

We will only show the first part of the theorem, $\probab_\alpha$ converges weakly to $\probab_1$ as $\alpha \uparrow 1$. This is the worst case scenario. The second statement of Theorem \ref{thm:mainthm} follows almost directly from the proof in this section.

 The proof of the upcoming Theorem \ref{thm:nnorm} uses coupling, which allows us to put both noises $\eta_\alpha$ and $\eta$ on the same probability space. This idea was introduced in \cite{davarchaoticstochastic} and lets us write our noise $\eta_\alpha$, for every $\alpha \in (0,1)$ in terms of one white noise $\eta$ with covariance 
\begin{align*}
\expectation \left[ \eta(t,x) \eta(s,y) \right] = \delta(t-s) \delta(x-y).
\end{align*}
The idea of coupling, or smoothing the noise in the spatial variable is not new. Authors in \cite{bertinicancrini} smoothed the noise in the spatial variable by an infinitely differentiable function with compact support. They have showed that this kind of smoothing converges to the heat equation with white noise  as our smoothing function converges to $\delta$ distribution. By coupling we mean that the martingale measure $\eta_\alpha$ will be defined as
\begin{align*}
\eta_\alpha([0,t] \times A) = \int_0^t \int_\R (\mathbbm{1}_A * h_\alpha)(x) \eta(ds,dx)\\
\end{align*}
and $f_\alpha$ in \eqref{eq:noisecolor} will be of the form 
\begin{align*}
f_\alpha (x) = (h_\alpha * h_\alpha)(x),
\end{align*}
where convolution is understood as convolution of generalized functions. Function $h_\alpha$ must be

\begin{align*}
h_\alpha (x) = c_{\frac{1-\alpha}{2} } g_{\frac{1+\alpha}{2} } (x) = \hat{g}_{\frac{1-\alpha}{2} }(x),
\end{align*}
since 
\begin{align*}
g_{1-\alpha}(\xi) = g_{\frac{1-\alpha}{2}}(\xi)\cdot g_{\frac{1-\alpha}{2}}(\xi) .
\end{align*}
One might notice that $h_\alpha\notin L^2(\R)$, but $\eta_\alpha$ is a well defined martingale measure, we refer the reader to \cite{davarchaoticstochastic} for more details. Before we state the main theorem of this section, let us state a technical lemma.

\begin{lemma} \label{lem:rieszkernelgaussdistrib}\cite[3.478]{tableofintegrals}\label{eq:boundwithgammafun}
The following equality holds for $s>0$ and $\gamma \in [0,1)$
\begin{align}\label{eq:negativemoment}
\int_\R \abs{x}^{-\gamma} e^{-s 4 \pi^2 x^2} dx = \left( \frac{1}{s 4\pi^2}\right)^{-(\gamma-1)/2}\Gamma(-\gamma/2+1/2).
\end{align}
\end{lemma}

In the rest of this section, we will prove the following main theorem
\begin{theorem}\label{thm:nnorm}
For every $k\geq 2$ we can find $\gamma$ such that 
\begin{align*}
\lim_{\alpha \uparrow 1} \mathcal{N}_{\gamma,k}(u_\alpha - u) = 0,
\end{align*}
where $\mathcal{N}_{\gamma,k}$ is the following norm \cite{davarcbms}

\begin{align*}
\mathcal{N}_{\gamma,k}(u) = \sup_{t\in [0,T]} \sup_{x\in \R} \left( e^{-\gamma t} \norm{u_t}_{L^k(P)}\right)~~,~~k\geq 2.
\end{align*}
\end{theorem}

Take the constant $T$ that appears in Theorem \ref{thm:mainthm} and the definition of the norm $\mathcal{N}_{\gamma,k}$ to be fixed throughout the whole proof. We will start our proof with Picard iterations for both noises $\eta_\alpha$ and $\eta$

\begin{align*}
u^{(n+1)}_t (y) &=  (u_0*p_t)(y) + \int_0^t \int_\R p_{t-s}(x-y) \sigma( u^{(n)}_s) (x) \eta(ds,dx)\\
u^{(n+1)}_{\alpha,t}(y) &= (u_0*p_t)(y) + \int_0^t \int_\R p_{t-s}(x-y) \sigma ( u^{(n)}_{\alpha,s}(x)) \eta_\alpha(ds,dx),
\end{align*}
which is equivalent to the following, thanks to \cite[Sec. 3.2]{davarchaoticstochastic}:
\begin{align*}
\textstyle u^{(n+1)}_t (y) &=  (u_0*p_t)(y) + \int_0^t \int_\R (p_{t-s}(\cdot-y)\sigma( u^{(n)}_s (\cdot))*\delta)(x) \eta(ds,dx)\\
\textstyle u^{(n+1)}_{\alpha,t}(y) &= (u_0*p_t)(y) + \int_0^t \int_\R (p_{t-s}(\cdot-y) \sigma( u^{(n)}_{\alpha,s}(\cdot))*h_\alpha)(x) \eta(ds,dx).
\end{align*}

First, let us estimate the $L^k(P)$ norm of the difference of Picard iterates $u^{(n+1)}_{\alpha,t}(y)- u^{(n+1)}_t (y)$,

\begin{align}\label{eq:firstpicarditerate}
&\expectation \left[\left| u^{(n+1)}_{\alpha,t}(y) - u^{(n+1)}_t (y) \right|^k \right] \\ 
&= \expectation \left[ \left| \int_0^t \int_\R \left( (p_{t-s}(\cdot-y)\sigma( u^{(n)}_{\alpha,s})*h_\alpha)(x) - (p_{t-s}(\cdot-y) \sigma( u^{(n)}_s )*\delta)(x)\right) \eta(ds,dx)   \right|^k \right]. \nonumber
\end{align}
Adding and subtracting the following term, $(p_{t-s}(\cdot-y)\sigma( u^{(n)}_s)*h_\alpha)(x)$ inside the integral and using the inequality $|a-b|^k \leq 2^k |a|^k + 2^k |b|^k$  yields

\begin{align*}
&\expectation \left[\left| u^{(n+1)}_{\alpha,t}(y) - u^{(n+1)}_t (y) \right|^k \right] \\
&\leq 2^k \expectation \left[ \left|\int_0^t \int_\R  (p_{t-s}(\cdot-y) (\sigma(u^{(n)}_{\alpha,s}) - \sigma(u^{(n) }_s))*h_\alpha)(x)  \eta(ds,dx)   \right|^k \right]\\
&~~~+2^k\expectation \left[ \left|\int_0^t \int_\R  (p_{t-s}(\cdot-y) \sigma(u^{(n) }_s)*(h_\alpha-\delta))(x)  \eta(ds,dx)   \right|^k \right].\\
\end{align*}
The next series of steps will be used multiple times throughout this work. First we will use Burkholder-Davis-Gundy inequality and Minkowski integral inequality. Burkholder-Davis-Gundy (BDG) inequality (see for example \cite[Thm. B.1]{davarcbms}) states that for any continuous $L^2$ martingale $M_t$ and $k\geq 2$ we have $\norm{M_t}^2_k \leq 4k \norm{\langle M \rangle_t}_{k/2}$, where $\langle M \rangle_t$ denotes the quadratic variation of $M$. Applying this inequality and evaluating the quadratic variation term \cite[Thm. 5.26]{davarminicourse} on both terms gives us

\begin{multline*}
\expectation \left[\left| u^{(n+1)}_{\alpha,t}(y) - u^{(n+1)}_t (y) \right|^k \right] \\
\leq \mathrm{const}\cdot \expectation\left[ \left( {\int_{[0,t]\times \R^2}  p_{t-s}(x-y)  \mathfrak{v}^{(n)}_s(x,z) f_\alpha(x-z) p_{t-s}(z-y)}  dsdxdz\right)^{k/2} \right]\\
+\mathrm{const} \cdot \expectation \Bigg[ \Bigg( \int_{[0,t]\times \R^2}  p_{t-s}(x-y)(z) \sigma(u_s^{(n)}(x))\left( f_\alpha - 2 h_\alpha + \delta \right)(x-z) \\
 p_{t-s}(z-y)\sigma(u_s^{(n)}(z))  dsdxdz \Bigg)^{k/2} \Bigg] ,
\end{multline*}
where 
\begin{align*}
\mathfrak{v}^{(n)}_s(x,z)= (\sigma(u^{(n)}_{\alpha,s}) - \sigma(u^{(n) }_s))(x)(\sigma( u^{(n)}_{\alpha,s}) - \sigma(u^{(n) }_s))(z).
\end{align*}
Minkowski integral inequality states that $\norm{\int f d\mu}_k \leq \int \norm{f}_kd\mu $ for any $\sigma$-finite measure $\mu$ on $\R^m, m\in \N$ and jointly measurable positive function $f$. We use this inequality on the first term in to order to obtain
\begin{multline*}
\expectation \left[\left| u^{(n+1)}_{\alpha,t}(y) - u^{(n+1)}_t (y) \right|^k \right] \\
\leq \mathrm{const}\cdot \left(\int_{[0,t]\times \R^2}  p_{t-s}(x-y)  v_s^{(n)} (x,z) f_\alpha(x-z) p_{t-s}(z-y)  dsdxdz\right)^{k/2} \\
+\mathrm{const}\cdot  \expectation \Bigg[ \Bigg( \int_{[0,t]\times \R^2}  p_{t-s}(x-y)(z) \sigma(u_s^{(n)}(x))\left( f_\alpha - 2 h_\alpha + \delta \right)(x-z) \\ p_{t-s}(z-y)\sigma(u_s^{(n)}(z))  dsdxdz \Bigg)^{k/2} \Bigg], 
\end{multline*}

\noindent
where $v_s^{(n)}$ denotes 
\begin{align*}
&v_s^{(n)} (x,z) = \expectation\left[\abs{(\sigma(u^{(n)}_{\alpha,s}) - \sigma(u^{(n) }_s))(x)}^{k/2}\abs{(\sigma( u^{\alpha,(n)}_s) - \sigma(u^{(n) }_s))(z)}^{k/2}\right]^{2/k}.\\
\end{align*}
Ultimately, we would like to show that $u_\alpha$ is close to $u$ as $\alpha \uparrow 1$. The original estimate for Picard iterates splits into two terms $\mathfrak{A}$ and $\mathfrak{B}$ where 

\begin{multline*}
\mathfrak{A}=\left(\int_{[0,t]\times \R^2}  p_{t-s}(x-y)  v_s^{(n)} (x,z) f_\alpha(x-z) p_{t-s}(z-y)  dsdxdz\right)^{k/2}\\
\end{multline*}
\begin{multline*}
\mathfrak{B} = \expectation \Bigg[\Bigg( \int_{[0,t]\times \R^2}  p_{t-s}(x-y)(z) \sigma(u_s^{(n)}(x))\left( f_\alpha - 2 h_\alpha + \delta \right)(x-z)  \\
p_{t-s}(z-y)
\sigma(u_s^{(n)}(z))  dsdxdz \Bigg)^{k/2} \Bigg].
\end{multline*}

For term $\mathfrak{A}$ we use Cauchy-Schwarz inequality and take supremum over the term involving expectation, which yields

\begin{multline*}
\mathfrak{A} \leq \Bigg( \int_0^t \sup_{x\in \R} \expectation\left[\abs{(\sigma(u_{\alpha,s}^{(n)}) -\sigma(u_s^{(n)}))(x)}^k\right]^{2/k} \int_{\R^2} 
 p_{t-s}(x-y)  f_\alpha(x-z) p_{t-s}(z-y)  dsdxdz \Bigg)^{k/2}.
\end{multline*}
The following identity holds
\begin{align}\label{eq:FTtrick}
\int_\R \int_\R \varphi(x) f_\alpha(x-y) \varphi(y) dx dy =\int_\R   f_\alpha(x) (\varphi * \tilde{\varphi})(x) dx  = \int_\R g_{1-\alpha} (\xi) |\mathcal{F}\varphi(\xi)|^2 d\xi,
\end{align}
for any $\varphi$ from Schwartz space $\mathcal{S}(\R)$ of rapidly decreasing test functions. This is a consequence of elementary properties of Fourier transform. We can further rewrite $\mathfrak{A}$ using identity \eqref{eq:FTtrick} and the assumption that $\sigma$ is Lipschitz continuous with Lipschitz constant $K$, that is $\abs{\sigma(x)-\sigma(y)} \leq K \abs{x-y}$

\begin{align*}
\mathfrak{A} &\leq \left( \int_0^t \sup_{x\in \R} \expectation\left[\abs{(\sigma(u_{\alpha,s}^{(n)}) -\sigma(u_s^{(n)}))(x)}^k\right]^{2/k}\int_\R g_{1-\alpha}(\xi) \abs{\hat p_{t-s}(\xi)}^2 ds d\xi \right)^{k/2}\\ 
&\leq K^{k}\left( \int_0^t \sup_{x\in \R} \expectation\left[\abs{(u_{\alpha,s}^{(n)} -u_s^{(n)})(x)}^k\right]^{2/k}\int_\R g_{1-\alpha}(\xi) \abs{\hat p_{t-s}(\xi)}^2 ds d\xi \right)^{k/2}.
\end{align*}
Multiply by term $e^{-k\gamma t}$ and obtain $\mathcal{N}_{\gamma,k}$ norm in the estimate
\begin{align*}
e^{-k\gamma t} \mathfrak{A} &\leq K^k\left( \int_0^t  e^{-2\gamma s}\sup_{x\in \R} \expectation\left[\abs{(u_{\alpha,s}^{(n)} -u_s^{(n)})(x)}^k\right]^{2/k}e^{-2\gamma(t-s)}\int_\R g_{1-\alpha}(\xi) \abs{\hat p_{t-s}(\xi)}^2 ds d\xi \right)^{k/2}\\
&\leq K^k\mathcal{N}_{\gamma,k}(u_\alpha^{n}-u^{n})^k\left( \int_0^t  e^{-2\gamma(t-s)}\int_\R g_{1-\alpha}(\xi) \abs{\hat p_{t-s}(\xi)}^2 ds d\xi \right)^{k/2}\\
&\leq K^k\mathcal{N}_{\gamma,k}(u_\alpha^{n}-u^{n})^k\left( \int_0^t e^{-2\gamma (t-s)} \int_\R \frac{1}{\abs{\xi}^{1-\alpha}} e^{-(t-s) \kappa 4 \pi^2 \xi^2} d\xi ds \right)^{k/2}.
\end{align*}
\noindent

Later on, we will see that we can make the integral on right hand side arbitrarily small. The estimate for $\mathfrak{B}$ uses a similar technique as the estimate for $\mathfrak{A}$, but some extra work is required because of the term $(f_\alpha - 2h_\alpha + \delta)$ inside the integral is not a positive function.  Thanks to \cite{dalangmainsrc},\cite[Cor. 3.4]{foondundavar} identity \eqref{eq:FTtrick} extends to a much broader class of functions. We will use this identity to bound term $\mathfrak{B}$. 
Quantity $\sigma(u_s^{(n)}(\cdot))p_{t-s}(\cdot-y) \in L^2(\R)\cap L^1(\R)$ 
almost surely, because 
\begin{align*}
\expectation\left[\norm{\sigma(u_s^{n}(\cdot))p_{t-s}(\cdot-y)}^2_{L^2(\R)}\right]\leq 2K^2(1+\mathcal{N}_{\gamma,2}(u^{(n)})^2)\norm{p_{t-s}(\cdot)}^2_{L^2(\R)} ~,
\end{align*}
for some $\gamma > 1$ and $\mathcal{N}_{\gamma,2}(u^{(n)})$ is uniformly bounded for every $n\in \N$ \cite{davarcbms}. A similar reasoning applies for $\norm{\cdot}_{L^1(\R)}$. We can write
\begin{align*}
&\expectation \left[\left( \int_{[0,t]\times\R^2}  p_{t-s}(x-y) \sigma(u_s^{(n)}(x))\left( f_\alpha - 2 h_\alpha + \delta \right)(x-z)  p_{t-s}(z-y)\sigma(u_s^{(n)}(z))  dsdxdz  \right)^{k/2}\right]\\
&= \expectation \left[\left( \int_{[0,t]\times \R} (g_{1-\alpha} -2 g_{\frac{1-\alpha}{2}} + 1)(\xi) \abs{\mathcal{F}\left( p_{t-s}(\cdot-y)\sigma(u_s^{(n)}(\cdot)) \right)(\xi)}^2  d\xi ds \right)^{k/2}\right].
\end{align*}
Split this integral into two parts, and use inequality $\abs{a-b}^{k/2} \leq 2^{k/2}\abs{a}^{k/2} + 2^{k/2}\abs{b}^{k/2}$ to get
\begin{align*}
&\expectation \left[ \left(\int_{[0,t]\times \R} (g_{1-\alpha} -2 g_{(1-\alpha)/2} + 1)(\xi) \abs{\mathcal{F}\left( p_{t-s}(\cdot-y)\sigma(u_s^{(n)}(\cdot)) \right)(\xi)}^2  d\xi ds \right)^{k/2} \right] \\
&\leq \mathrm{const}\cdot (\mathfrak{C} + \mathfrak{D})~,
\end{align*}
where
\begin{align*}
\mathfrak{C} &= \expectation \left[ \left( \int_{[0,t]}\int_{[-1,1]} (g_{1-\alpha} -2 g_{\frac{1-\alpha}{2}} + 1)(\xi) \abs{\mathcal{F}\left( p_{t-s}(\cdot-y)\sigma(u_s^{(n)}(\cdot)) \right)(\xi)}^2  d\xi ds\right)^{k/2} \right] \\
\mathfrak{D} &= \expectation \left[ \left(\int_{[0,t]}\int_{\R \setminus [-1,1]} (g_{1-\alpha} -2 g_{\frac{1-\alpha}{2}} + 1)(\xi) \abs{\mathcal{F}\left( p_{t-s}(\cdot-y)\sigma(u_s^{(n)}(\cdot)) \right)(\xi)}^2  d\xi ds\right)^{k/2}\right] .
\end{align*}
Properties of Fourier transform and Lipschitz continuity of $\sigma(x)$ give us 
\begin{align}\label{eq:insideft}
&\abs{\mathcal{F}\left( p_{t-s}(\cdot-y)\sigma(u_s^{(n)}(\cdot)) \right)(\xi)}^2
\leq \norm{p_{t-s}(\cdot - y)\sigma(u_s^{(n)}(\cdot))}_{L^1(\R)}^2 \nonumber\\
&\leq K^2\norm{p_{t-s}(\cdot - y)(1+|u_s^{(n)}(\cdot)|)}_{L^1(\R)}^2 
\leq K^2(2 + 2 \norm{p_{t-s}(\cdot - y)u_s^{(n)}(\cdot)}_{L^1(\R)}^2).
\end{align}
for the term inside of $\mathfrak{C}$. 	Splitting the term \eqref{eq:insideft} inside of the integral into two yields
\begin{align*}
\mathfrak{C} &\leq \expectation \left[ \left(\int_{[0,t]}\int_{[-1,1]} (g_{1-\alpha} -2 g_{\frac{1-\alpha}{2}} + 1)(\xi) K^2(2 + 2 \norm{p_{t-s}(\cdot - y)u_s^{(n)}(\cdot)}_{L^1(\R)}^2)  d\xi ds \right)^{k/2}\right]\\
&\leq C_\alpha + \mathrm{const} \cdot\expectation \left[\left( \int_{[0,t]}\int_{[-1,1]} (g_{1-\alpha} -2 g_{\frac{1-\alpha}{2}} + 1)(\xi)  \norm{p_{t-s}(\cdot - y)u_s^{(n)}(\cdot)}_{L^1(\R)}^2  d\xi ds\right)^{k/2} \right] ,
\end{align*}
where $C_\alpha$ denotes 
\begin{align*}
C_\alpha =  \mathrm{const} \left(\int_{[0,t]}\int_{[-1,1]} (g_{1-\alpha} -2 g_{\frac{1-\alpha}{2}} + 1)(\xi)d\xi ds \right)^{k/2}.
\end{align*}
Term $C_\alpha$ can be made as small as we like, due to the dominated convergence theorem.  We use Minkowski integral inequality and get
\begin{align*}
\mathfrak{C}&\leq C_\alpha + \mathrm{const} \left( \int_{[0,t]} \sup_{x\in \R} \expectation\left[\abs{u_s^{(n)}(x)}^k\right]^{2/k}  \int_{[-1,1]} (g_{1-\alpha} -2 g_{\frac{1-\alpha}{2}} + 1)(\xi)    d\xi ds \right)^{k/2}\\
&\leq C_\alpha + \mathrm{const} \left( \int_{[0,t]} e^{-k\gamma s}\sup_{x\in \R} \expectation\left[\abs{u_s^{(n)}(x)}^k\right]^{2/k} e^{k\gamma s} \int_{[-1,1]} (g_{1-\alpha} -2 g_{\frac{1-\alpha}{2}} + 1)(\xi)    d\xi ds \right)^{k/2}\\
&\leq C_\alpha + \mathrm{const}\cdot \mathcal{N}_{\gamma,k}(u^{n})^k \left( \int_{[0,t]}  e^{k\gamma s} \int_{[-1,1]} (g_{1-\alpha} -2 g_{\frac{1-\alpha}{2}} + 1)(\xi)    d\xi ds \right)^{k/2}.
\end{align*}
From the general theory of stochastic partial differential equations \cite{davarcbms} we know that the term $\mathcal{N}_{\gamma,k}(u^k)$ is bounded uniformly for every $k$, for some $\gamma$ sufficiently large. The integral term bounding $\mathfrak{C}$ can be made arbitrarily small, again from the dominated convergence theorem.

All we have left to do is find the estimate for $\mathfrak{D}$. Add and subtract the term $\sigma(u_s(s))$ inside the Fourier transform, split into two integrals and obtain 
\begin{align*}
\mathfrak{D} 
&\leq  \mathrm{const}\cdot \sup_{n>n_0}\expectation \left[\left(  \int_0^t \int_\R \left(p_{t-s}(x-y) \left(\sigma(u_s^{(n)}(x)) - \sigma(u_s(x)) \right)  \right)^2 dx dt \right)^{k/2}\right] \\
&~~+\mathrm{const}\cdot \expectation\left[\left(\int_0^t \int_{\R \setminus [-1,1]} (g_{1-\alpha} -2 g_{\frac{1-\alpha}{2}} + 1)(\xi) \abs{\mathcal{F}\left( p_{t-s}(\cdot-y)\sigma(u_s(\cdot)   \right)(\xi)}^2  d\xi ds \right)^{k/2}\right],  
\end{align*}
for large $n>n_0$. The first integral can be made arbitrarily small, this is from convergence of Pickard's iterations and theory of SPDEs \cite{davarcbms}. We have used Plancherel's theorem and the fact that $(g_{1-\alpha} -2 g_{\frac{1-\alpha}{2}} + 1)$ is bounded by constants on $\R\setminus [-1,1]$, uniformly for all $\alpha \in (0,1)$. The second term can be made arbitrarily small  due to the dominated convergence theorem.  If $n$ is small, that is smaller than some $n_0$, then we need to pick appropriate  $\alpha$ close to one so that the term $\mathfrak{D}$ is arbitrarily close to zero for all $n<n_0$. This again follows from the dominated convergence theorem. The term $(g_{1-\alpha} - 2g_{\frac{1-\alpha}{2}} + 1)$ inside $\mathfrak{D}$ converges pointwise to zero.

The estimate of $e^{-k\gamma t} \mathfrak{A}$ is bounded by the $\mathcal{N}_{\gamma,k}$ norm of the previous Picard iterate. Both terms $\mathfrak{C}$ and $\mathfrak{D}$ can be made smaller than arbitrary $\epsilon$, thus term $\mathfrak{B}$ can be made smaller then arbitrary $\epsilon$. Add up all those estimates and multiply \eqref{eq:firstpicarditerate} by $e^{-k\gamma t}$ to get
\begin{multline*}
e^{-k\gamma t} \expectation \left[\left| u^{(n+1)}_{\alpha,t}(y) - u^{(n+1)}_t (y) \right|^k \right]\\
\leq 
\mathrm{const}\cdot \mathcal{N}_{\gamma,k}(u_\alpha^{(n)} - u^{(n)})^k \left(\int_0^t e^{-2\gamma (t-s)} \int_\R \frac{1}{\abs{\xi}^{1-\alpha}} e^{-(t-s) \kappa 4 \pi^2 \xi^2} d\xi ds \right)^{k/2} + \epsilon~, \\
\end{multline*}
where $\gamma \geq 1$. We use Lemma \ref{lem:rieszkernelgaussdistrib} to evaluate the term inside the integral. A straightforward calculation yields 
\begin{align*}
&e^{-k\gamma t} \expectation \left[\left| u^{(n+1)}_{\alpha,t}(y) - u^{(n+1)}_t (y) \right|^k \right]\\
&\leq \mathrm{const}\cdot \mathcal{N}_{\gamma,k}(u_\alpha^{(n)} - u^{(n)})^k \left(\int_0^t e^{-2\gamma (t-s)} (t-s)^{-\alpha/2} ds \right)^{k/2} + \epsilon\\
&\leq \mathrm{const}\cdot \mathcal{N}_{\gamma,k}(u_\alpha^{(n)} - u^{(n)})^k \left(\int_0^\infty e^{-2\gamma s} s^{-\alpha/2} ds\right)^{k/2} + \epsilon\\
&\leq \mathrm{const}\cdot \mathcal{N}_{\gamma,k}(u_\alpha^{(n)} - u^{(n)})^k \left( \left(\frac{1}{\gamma}\right)^{1-\alpha/2} \right)^{k/2} + \epsilon
\leq \mathrm{const}\cdot \mathcal{N}_{\gamma,k}(u_\alpha^{(n)} - u^{(n)})^k \left(\frac{1}{\gamma}\right)^{k/4} + \epsilon ~.
\end{align*}
We can take supremum over $y\in\R$ and $t \in [0,T]$ to get
\begin{align*}
\mathcal{N}_{\gamma,k} (u_\alpha^{(n+1)} - u^{(n+1)})^k  \leq \mathrm{const} \cdot \mathcal{N}_{\gamma,k}(u_\alpha^{(n)} - u^{(n)})^k \left(\frac{1}{\gamma}\right)^{k/4} + \epsilon ~.
\end{align*}

This defines a convergent geometric series assuming that the coefficient $C\left(\frac{1}{\gamma}\right)^{k/2} < 1$, where $C \equiv \mathrm{const}$. This also implies 
\begin{align*}
\mathcal{N}_{\gamma,k} (u_\alpha - u) \leq \frac{\epsilon}{1-\frac{C}{\gamma^{k/2}}}~,
\end{align*}
for some $\alpha$ close to 1. This concludes the proof of Theorem \ref{thm:nnorm}.

\subsection{Convergence of finite dimensional distributions}

Theorem \ref{thm:nnorm} also states that the solution $u_\alpha$ converges to $u$ in $L^2(P)$ norm for every $t \in [0,T]$ and $x\in \R$. This implies weak convergence of finite dimensional distributions of $u_\alpha$ to finite dimensional distributions of $u$. 
The easiest way to see that is to show convergence in probability for a finite number of pairs $(t_i,x_i) \in [0,T]\times \R$, which implies weak convergence of finite dimensional distribution \cite[pg. 27]{billingsleyconvergence}. The convergence in probability follows from Chebyshev's inequality and convergence in $L^2(P)$ norm.

\subsection{Estimates for Kolmogorov's continuity theorem and Tightness} 
We will prove tightness (and thus weak convergence) from Kolmogorov's continuity theorem. Before we begin the proof, we will need the following two lemmas.
\begin{lemma} \label{lem:spatialdifference} \cite[Lemma 6.4]{davarchaoticstochastic} For all $s>0$ and $x\in \R$
\begin{align*}
\int_\R \abs{ p_t (y-x) - p_t(y)} dy \leq \mathrm{const} \cdot \left( \frac{\abs{x}}{\sqrt{\kappa t }} \wedge 1\right),\\
\end{align*}
where the implied constant does not depend on $(s,x)$.
\end{lemma}
\begin{lemma}\label{lem:timedifference}
For all $t,\epsilon>0$ and $x\in \R$  we have
\begin{align*}
\int_\R \abs{ p_{t+\epsilon} (y) - p_t(y)} dy \leq \mathrm{const}\cdot \left(\left( \log(t+\epsilon) - \log(t)\right) \wedge 1\right) .
\end{align*}
\end{lemma}
\begin{proof} Direct computation gives us
\begin{align*}
\int_\R &\abs{ p_{t+\epsilon} (y) - p_t(y)} dy = \int_\R \abs{ \int_t^{t+\epsilon} \dot{p}_s(y) ds }dy = \int_\R \abs{ \int_t^{t+\epsilon}\left( -\frac{1}{2t} + \frac{y^2}{2t^2 \kappa} \right)p_s(y) ds }dy\\
&\leq  \int_t^{t+\epsilon} \int_\R \left( \frac{1}{2t} + \frac{y^2}{2t^2 \kappa} \right)p_s(y) dy ds 
= \int_t^{t+\epsilon} \frac{1}{t} dt = \left(\log(t+\epsilon) - \log(t) \right).
\end{align*}
In addition, we have that 
\begin{align*}
\int_\R \abs{ p_{t+\epsilon} (y) - p_t(y)} dy \leq 2.
\end{align*}
\end{proof}

\subsubsection{Difference in the spatial variable}\label{sec:differenceinx}
Let us estimate the difference in the spatial variable

\begin{align*}
\expectation\left[ \abs{u_{\alpha,t}(x) - u_{\alpha,t}(y)}^k \right] = \expectation \left[ \abs{\int_0^t \int_\R \left( p_{t-s}(x-z) - p_{t-s}(y-z) \right) \sigma(u_{\alpha,s}(z))\eta(ds,dz)}^k \right],
\end{align*}
and denote 
\begin{align*}
B(z) &=  \left( p_{t-s}(x-z) - p_{t-s}(y-z) \right)~,\\
A(x,y) &= \sigma(u_{\alpha,s}(x)) \sigma(u_{\alpha,s}(y)).
\end{align*}
We will proceed just as in Section \ref{sec:mainsectionwithproof}. We use BDG, Minkowski integral inequality, Cauchy-Schwarz inequality and take the absolute value inside the integral and get
\begin{align*}
&\expectation\left[ \abs{u_{\alpha,t}(x) - u_{\alpha,t}(y)}^k \right] \\
&\leq \mathrm{const}\cdot \expectation \left[ \abs{\int_0^t \int_\R \int_\R f_\alpha(z-w) B(z)B(w)A(x,y)  dsdzdw } ^{k/2} \right]\\
&\leq \mathrm{const} \abs{\int_0^t \sup_{x\in\R}\norm{\sigma(u_{\alpha,s}(x))}_k^2 \int_\R \int_\R f_\alpha(z-w) \abs{B(z)}\abs{B(w)}  dsdzdw }^{k/2}\\
&\leq \mathrm{const} \abs{\int_0^t \sup_{x\in\R}  \norm{\sigma(u_{\alpha,s}(x))}_k^2 (f_\alpha*p_{t-s})(0) \int_\R  \abs{p_{t-s}(x-z) - p_{t-s}(y-z)}   dsdz }^{k/2}\\
&\leq \mathrm{const} (1 + \mathcal{N}_{\gamma,k}(u_\alpha)^k)\abs{ \int_0^t (f_\alpha*p_{t-s})(0)\int_\R \abs{p_{t-s}(x-z) - p_{t-s}(y-z)} dz ds}^{k/2}\\
&\leq \mathrm{const} (1 + \mathcal{N}_{\gamma,k}(u_\alpha)^k) \abs{ \int_0^t (f_\alpha * p_{t-s} ) (0) \left( \frac{\norm{x-y}}{\sqrt{\kappa t}} \wedge 1  \right)}^{k/2},
\end{align*}
where the last inequality is due to Lemma \ref{lem:spatialdifference}. We also used the fact that 
\begin{align*}
\int_\R f_\alpha(z-w) \left| p_{t-s}(x-w) - p_{t-s}(y-w) \right|dw \leq 2\left(f_\alpha * p_{t-s} \right)(0).
\end{align*}
The inequality $r \wedge 1 \leq r^{2a}$ for $a\in (0,1/2)$ gives us
\begin{multline}\label{eq:spacediffestimate}
\expectation\left[ \abs{u_{\alpha,t}(x) - u_{\alpha,t}(y)}^k \right] \\
\leq \mathrm{const} (1 + \mathcal{N}_{\gamma,k}(u_\alpha)^k) \norm{x-y}^{ak} \abs{\int_0^t (f_\alpha * p_{t-s} ) (0)\cdot (t-s)^{-a} ds  }^{k/2}.
\end{multline}

It remains to show that the integral on the right hand side is bounded for all $\alpha \in (\alpha_0,1)$, $\alpha_0>0$. To show this,  we will need an explicit form of $f_\alpha$. The result is stated in the next lemma.
\begin{lemma}
For every $1>\alpha > \alpha_0 > 0$ we have  
\begin{align*}
f_\alpha*p_{s} (0) \leq \mathrm{const}\cdot s^{-\alpha/2},
\end{align*}
where the constant depends only on our choice of $\alpha_0$.
\end{lemma}
\begin{proof}
By direct computation and \eqref{eq:negativemoment} we get
\begin{align*}
(f_\alpha*p_{s})(0) = c_{1-\alpha} \int_\R  \frac{1}{\abs{x}^\alpha} p_s(x)  dx = 2 \frac{\sin\left(\frac{(1-\alpha) \pi}{2} \right)\Gamma(\alpha)}{(2\pi)^\alpha} 2^{-\alpha/2} \Gamma\left(\frac{1-\alpha}{2}\right) s^{-\alpha/2}\pi^{-1/2}.
\end{align*}
The boundedness of constant 
\begin{align*}
2 \frac{\sin\left(\frac{(1-\alpha) \pi}{2} \right)\Gamma(\alpha)}{(2\pi)^\alpha} 2^{-\alpha/2} \Gamma\left(\frac{1-\alpha}{2}\right) \pi^{-1/2}
\end{align*}
can be concluded from Euler's reflection formula.
\end{proof}
Because of the lemma above, the integral on the right hand side of \eqref{eq:spacediffestimate} is finite as long as $a\in(0,1/2)$.

\subsubsection{Difference in the time variable} \label{sec:differenceint}
The difference in the time variable is going to be, for $\delta > 0$ 
\begin{multline*}
\expectation\left[ \abs{u_{\alpha,t+\delta}(x) - u_{\alpha,t}(x)}^k \right] \\
=\mathrm{const}\cdot \expectation \left[ \abs{\int_0^t \int_\R \left( p_{t+\delta-s}(x-z) - p_{t-s}(x-z) \right) \sigma(u_{\alpha,s}(z))\eta(ds,dz)}^k \right]\\
+\mathrm{const}\cdot \expectation\left[ \abs{\int_t^{t+\delta} \int_\R p_{t+\delta-s}(x-z)\sigma(u_{\alpha,s}(z))\eta(ds,dz)}^k \right].
\end{multline*}
Let us estimate the second integral, we can use the same technique as in the case of the spatial variable and write 

\begin{align*}
&\expectation\left[ \abs{\int_t^{t+\delta} \int_\R p_{t+\delta-s}(x-z)\sigma(u_{\alpha,s}(z))\eta(ds,dz)}^k \right] \\
&\leq \mathrm{const} \left(
\int_t^{t+\delta} \sup_x \expectation\left[\abs{\sigma(u_{\alpha,s}(x))}^k\right] \int_\R \frac{1}{\abs{\xi}^{1-\alpha}} \abs{\hat p_{t+\delta-s} (\xi)}^2 d\xi ds \right)^{k/2}\\
&\leq \mathrm{const} (1+\mathcal{N}_{\gamma,k}(u_\alpha)^k) \left( \int_t^{t+\delta} (t+\delta-s)^{-\alpha/2} ds\right)^{k/2}\\
&\leq  \mathrm{const} (1+\mathcal{N}_{\gamma,k}(u_\alpha)^k) \abs{\delta}^{k(2-\alpha)/4}\leq  \mathrm{const} (1+\mathcal{N}_{\gamma,k}(u_\alpha)^k) \abs{\delta}^{k/4}.
\end{align*}
The estimate for the second integral will be
\begin{align}
&\expectation \left[ \abs{\int_0^t \int_\R \left( p_{t+\delta-s}(x-z) - p_{t-s}(x-z) \right) \sigma(u_{\alpha,s}(z))\eta(ds,dz)}^k \right] \nonumber\\
&\leq \mathrm{const} \left( \int_0^t \sup_{x\in R}\expectation\left[\abs{\sigma(u_{\alpha,s}(x))}^k\right]^{2/k}  (f_\alpha * p_{t-s})(0)\int_\R \abs{p_{s+\delta}(z) - p_{s}(z)}  dz ds  \right)^{k/2} \nonumber\\
&\leq \mathrm{const} (1+\mathcal{N}_{\gamma,k}(u_\alpha)^k) \left( \int_0^t  (f_\alpha * p_{t-s})(0)\int_\R \abs{p_{s+\delta}(z) - p_{s}(z)}  dz ds  \right)^{k/2} \nonumber \\
&\leq \mathrm{const} (1+\mathcal{N}_{\gamma,k}(u_\alpha)^k) \left( \int_0^t  s^{-1/2} \left(\log(s+\delta)-\log(s)  \right) ds  \right)^{k/2} \nonumber \\
&\leq \mathrm{const} (1+\mathcal{N}_{\gamma,k}(u_\alpha)^k) \left( 4\sqrt{\delta}\operatorname{atan}\left( \sqrt{\frac{t}{\delta}}\right) + 2\sqrt{t}\log(1+\delta/t)\right)^{k/2} ,\label{eq:timedifffirstintegralestimate}
\end{align}
by using a similar technique as in the case for the spatial variable and Lemma \ref{lem:timedifference}. The inequality $\log(1+\zeta) < \sqrt{\zeta}$ for all $\zeta>0$ gives us 
\begin{multline}
\expectation \left[ \abs{\int_0^t \int_\R \left( p_{t+\delta-s}(x-z) - p_{t-s}(x-z) \right) \sigma(u_{\alpha,s}(z))\eta(ds,dz)}^k \right] \\
\leq \mathrm{const} \left(1+\left[\mathcal{N}_{\gamma,k}(u_\alpha)\right]^k\right) \delta^{k/4}. \label{eq:timediffsecondintegralestimate}
\end{multline}
We can combine both estimates \eqref{eq:timedifffirstintegralestimate} and  \eqref{eq:timediffsecondintegralestimate} to finally get 
\begin{align*}
\expectation\left[ \abs{u_{\alpha,t+\delta}(x) - u_{\alpha,t}(x)}^k \right] &=\mathrm{const}\cdot\left(1+\left[\mathcal{N}_{\gamma,k}(u_\alpha)\right]^k\right) \delta^{k/4}~.
\end{align*}

\subsection{Kolmogorov's continuity theorem and tightness} \label{sec:kolmogorovcontinuity}
Let us mention that $N_{k,\gamma}(u_\alpha)$ is finite \cite[Prop. 9.1]{davarchaoticstochastic} for every choice of $\alpha \in (0,1)$. We also know that $u_\alpha$ is continuous in $\mathcal{N}_{\gamma,k}$ norm thanks to Theorem \ref{thm:nnormalpha0}. Those two facts together imply that $\mathcal{N}_{\gamma,k}(u_\alpha)$ is uniformly bounded for $\alpha\in (\alpha_0,1),\alpha_0>0$. We would like to emphasize that we are not creating a circular argument, since the proof of  Theorem \ref{thm:nnormalpha0} solely depends on the proof of Theorem \ref{thm:nnorm}. We have that for every $1>\alpha>\alpha_0>0$ and $(s,x),(t,y)$ from $[0,T]\times[-N,N]\subset \R^+_0 \times \R$ the following holds for $k\geq 2$	

\begin{align*}
\expectation\left[ \abs{u_{\alpha,s}(x) - u_{\alpha,t}(y)}^{k}\right] \leq \mathrm{const} \abs{x-y}^{ka} + \mathrm{const} \abs{t-s}^{k/4},
\end{align*}
where $a\in (0,1/2)$, thanks to our estimates from sections \ref{sec:differenceinx}  and \ref{sec:differenceint}.
Denote $\rho(t,x) = \abs{x}^a + \abs{t}^{1/4}$, then Kolmogorov's continuity theorem states that there is a modification of $u_{\alpha,s}(x)$ such that (see for example \cite[pg. 113]{davarcbms})
\begin{align}\label{eq:kolmogorovcontinuity}
\expectation\left[\sup_{(s,x),(t,y)\in D} \abs{\frac{{u_{\alpha,s}(x) - u_{\alpha,t}(y)}}{\rho(s-t,x-y)^q}}^k \right] < \Lambda < +\infty
\end{align}
for every $\alpha \in (\alpha_0,1)$ and $q\in (0,1-H/k)$ where $H=1/a+4$. By Chebyshev's inequality and  \eqref{eq:kolmogorovcontinuity}, we can write

\begin{align*}
\probab \left\{ \sup_{\substack{(s,x),(t,y)\in D\\ \rho(s-t,x-y) < \delta }} \abs{u_{\alpha,s}(x) - u_{\alpha,t}(y)} > \epsilon\right\} < \frac{\Lambda}{\epsilon^k} \delta^{kq} ~,
\end{align*}
which implies 
\begin{align*}
\lim_{\delta \rightarrow 0} \sup_{\alpha \in (\alpha_0,1)} \probab \left\{ \sup_{\substack{(s,x),(t,y)\in D\\ \rho(s-t,x-y) < \delta }} \abs{u_{\alpha,s}(x) - u_{\alpha,t}(y)} > \epsilon\right\} = 0 ~,
\end{align*}
for every $\epsilon > 0$. We can conclude \cite[Thm. 2]{wichuraWeakConvergence} that the solution $u_\alpha$ converges weakly to $u$ in $\mathcal{C}$.

\section{Continuity in $\alpha$ for $\alpha \in (0,1)$}

Our proof in section \ref{sec:mainsectionwithproof} also implies continuity in $\alpha$ for $\alpha \in (0,1)$. We will only comment on how the proof would change in section \ref{sec:mainsectionwithproof}.

Results in sections \ref{sec:differenceint}, \ref{sec:differenceinx} and \ref{sec:kolmogorovcontinuity} hold without change. The piece of proof that needs to be slightly modified is the proof of Theorem \ref{thm:nnorm}. Recreating Theorem \ref{thm:nnorm} for the new setting would not give us any new technique or insight. Let us simply state the new Theorem.

\begin{theorem}\label{thm:nnormalpha0}
For every $k\geq 2$ and $\alpha_0 \in (0,1)$ we can find $\gamma$ such that
\begin{align*}
\lim_{\alpha \rightarrow \alpha_0} \mathcal{N}_{\gamma,k}(u_\alpha - u_{\alpha_0}) = 0.
\end{align*} 
\end{theorem}

The proof of Theorem 3 follows the same general direction of the proof of Theorem \ref{thm:nnorm} with the following changes. We need to replace $u_t(x)$ with $u_{\alpha_0,t}(x)$ and change $\left( f_{\alpha} - 2h_\alpha + \delta\right)$ in the estimate for $\mathfrak{B}$ to $\left(f_{\alpha} -2 f_{\frac{\alpha+\alpha_0}{2}}+ f_{\alpha_0}\right) $ and change $\left(g_{1-\alpha} -2g_{\frac{1-\alpha}{2}} + 1\right)$ to 
$\left(g_{1-\alpha} -2g_{1-\frac{\alpha + \alpha_0}{2}} +g_{1-\alpha_0}\right)$.

\section*{Acknowledgement}
This article was submitted in partial fulfillment of the requirements for the degree of Doctor of Philosophy in Mathematics at the University of Utah under the supervision of Professor Davar Khoshnevisan. 
This research was supported in part by the NSFs grant DMS-1307470.

\bibliographystyle{plain}
\bibliography{refs}

\end{document}